\documentclass[11pt,reqno,a4paper]{amsart}

\usepackage{amsmath}
\usepackage{enumerate}
\usepackage{amssymb}
\usepackage{amscd}
\usepackage{amsthm}
\usepackage{amsfonts}
\usepackage{graphicx}
\usepackage{hyperref}
\usepackage{xcolor}

\usepackage{mathabx}

\usepackage{enumitem}

\usepackage{ccfonts,eulervm} 
\usepackage[T1]{fontenc}

\numberwithin{equation}{section}

\newcommand{\SL}{\operatorname{SL}}

\newcommand{\SO}{\operatorname{SO}}

\newcommand{\GL}{\operatorname{GL}}

\newcommand{\cC}{\mathcal{C}}

\newcommand{\cG}{\mathcal{G}}
\newcommand{\cH}{\mathcal{H}}

\newcommand{\cP}{\mathcal{P}}

\newcommand{\cR}{\mathcal{R}}

\newcommand{\bC}{\mathbb{C}}

\newcommand{\bN}{\mathbb{N}}

\newcommand{\bQ}{\mathbb{Q}}
\newcommand{\bR}{\mathbb{R}}

\newcommand{\bT}{\mathbb{T}}

\newcommand{\bZ}{\mathbb{Z}}
\newcommand{\ra}{\rightarrow}

\newcommand{\qand}{\quad \textrm{and} \quad}

\newcommand\subsetsim{\mathrel{%
\ooalign{\raise0.2ex\hbox{$\subset$}\cr\hidewidth\raise-0.8ex\hbox{\scalebox{0.9}{$\sim$}}\hidewidth\cr}}}
\newcommand{\eps}{\varepsilon}

\DeclareMathOperator{\Ad}{Ad}

\DeclareMathOperator{\trace}{tr}

\DeclareMathOperator{\Sym}{Sym}
\DeclareMathOperator{\Stab}{Stab}

\DeclareMathOperator{\Leb}{Leb}

\theoremstyle{theorem}
\newtheorem{theorem}{Theorem}[section]
\newtheorem{corollary}[theorem]{Corollary}
\newtheorem{proposition}[theorem]{Proposition}
\newtheorem{lemma}[theorem]{Lemma}

\theoremstyle{definition}
\newtheorem{definition}[theorem]{Definition}
\newtheorem{remark}[theorem]{Remark}
\newtheorem*{example}{Example}

\theoremstyle{exercise}

\begin{document}

\title{Twisted patterns in large subsets of $\bZ^N$}

%  Author I information
\author{Michael Bj\"orklund}
\address{Department of Mathematics, Chalmers, Gothenburg, Sweden}
\email{micbjo@chalmers.se}

\author{Kamil Bulinski}
\address{School of Mathematics and Statistics,
University of Sydney, Australia}
\email{K.Bulinski@maths.usyd.edu.au}

\thanks{}

\keywords{Multiple recurrence, equidistribution, invariants}

\subjclass[2010]{Primary: 37A45; Secondary: 11P99, 37A30}

\date{}

\dedicatory{}

\maketitle

\begin{abstract}
Let $E \subset \bZ^N$ be a set of positive upper Banach density and let $\Gamma < \GL_N(\bZ)$ be a finitely generated, strongly irreducible subgroup whose Zariski closure in $\GL_N(\bR)$ is a Zariski connected semisimple
group with no compact factors. Let $Y$ be any set and suppose that $\Psi : \bZ^N \ra Y$ is a $\Gamma$-invariant
function. We prove that for every positive integer $m$, there exists a positive integer $k$ with the property that for
every finite set $F \subset \bZ^N$ with $|F| = m$, we have 
\[
\Psi(kF) \subset \Psi(E-b) \quad \textrm{for some $b \in E$}.
\]
Furthermore, if $E$ is an aperiodic Bohr$_o$-set, we can choose $k = 1$ and $b = 0$. As one of many applications 
of this result, we show that if $E_o \subset \bZ$ has positive upper Banach density, then, for any integer $m$, there 
exists an integer $k$ with the property for \emph{every} finite set $F \subset \bZ$, we can find $x,y,z \in E_o$ such
that
\[
k^2  F \subset \big\{ (u-x)^2 + (v-y)^2 - (w-z)^2 \, : \, u,v,w \in E_o \big\}.
\]
In particular, if $E_o \subset \bZ$ is an aperiodic Bohr$_o$-set, then every integer can be written on the form 
$u^2 + v^2 - w^2$ for some $u,v,w \in E_o$. Our techniques
use recent results by Benoist-Quint and Bourgain-Furman-Lindenstrauss-Mozes on equidistribution of random walks 
on automorphism groups of tori.
\end{abstract}

\section{Introduction}

We begin by recalling the following classical result of Furstenberg and Katznelson \cite{FKmulti}. The upper Banach density of a subset $E \subset \bZ^N$ will be defined in Appendix I.

\begin{theorem}Suppose that 
$E \subset \bZ^N$ has positive upper Banach density.
Then, for every finite set $F \subset \bZ^N$, there exists a positive integer $k$ such that
\begin{equation}
\label{fk}
kF \subset E - b, \quad \textrm{for some $b \in E$}.
\end{equation} 
The case $N = 1$ corresponds to Szemer\'edi's celebrated theorem on arithmetic progressions. 
\end{theorem} 

This is an archetypal result in Arithmetic Ramsey theory. We stress the order of the quantifiers; 
\emph{the integer $k$ heavily depends on the finite set $F$}. In this paper we shall prove a 
"twisted" analogue of Furstenberg-Katznelson's Theorem, for which the dependence between 
the integer $k$ and the set $F$ disappears. To motivate this line of study, we begin by giving three applications.

\subsection{Quadratic forms}
A very influential result in Geometric Ramsey theory by Furstenberg, Katznelson and Weiss \cite{FKW} 
asserts that if $E \subset \bR^N$ is a Borel set with positive density in the sense that
\[
\limsup_{R \ra \infty} \frac{\Leb(E \cap B(R))}{R^N} > 0,
\]
where $\Leb$ denotes the Lebesgue measure on $\bR^N$ and $B(R)$ denotes the Euclidean ball of radius $R$ around the origin, then there exists $R_o > 0$ such that
\[
D(E) = \big\{ \|x-y\|^2 \, : \, x, y \in E \big\} \supset [R_o,\infty),
\]
where $\|\cdot\|$ denotes the Euclidean norm on $\bR^N$. In other words, all sufficiently large Euclidean distances are realized within the set $E$. Recently, Magyar 
\cite{Magyar} established the following discrete analogue of this phenomenon. 

\begin{theorem}[Theorem 1, \cite{Magyar}] Fix an integer $N \geq 5$ and let
\[ 
Q(x_1,\ldots,x_N) = x_1^2 + \ldots + x_N^2.
\]
Then, for every subset $E \subset \bZ^N$ of positive upper Banach density, there exist positive 
integers $R_o$ and $k$ such that
\[
k^2 \mathbb{Z} \cap \left[R_o, \infty \right) \subset Q(E-E).
\]
\end{theorem}

Our first application consists of an analogue of Magyar's result for \emph{indefinite} quadratic forms. 
Contrary to Magyar's result, we focus here \emph{not} on the values of $Q$ restricted to a 
\emph{difference set} of a set $E \subset \bZ^N$ of positive upper Banach density, but rather 
we study the values of $Q$ restricted to \emph{some} translate of the set $E$. We stress that 
our techniques do not apply to the quadratic forms in Magyar's Theorem as the (real points) of
the symmetry group $\SO(N)$ is compact. For the notion of an (aperiodic) 
Bohr set we refer the reader to Section \ref{Bohr}.

\begin{theorem}
\label{cons3}
Let $p,q \geq 1$ and $p+q \geq 3$ and $E \subset \bZ^{p+q}$ a set of positive upper Banach
density. Let $Q$ denote the quadratic form on $\bR^{p+q}$ defined by
\[
Q(z) = \sum_{i=1}^p \mu_i x_i^2 - \sum_{j=1}^q \lambda_j y_j^2, 
\quad 
\textrm{for $z = (x_1,\ldots,x_p,y_1,\ldots, y_q) \in \bR^{p+q}$},
\]
where $\mu_1,\ldots,\mu_p$ and $\lambda_1,\ldots,\lambda_q$ are positive integers. Let 
$m$ be a positive integer. Then there exists a positive integer $k$ with the property that
for every finite subset $F \subset \bZ^{p+q}$ with $|F| = m$, we have
\[
k^2 \, Q(F) \subset Q(E-b), \quad \textrm{for some $b \in E$}.
\]
If $E$ is an aperiodic Bohr set, then $k$ can be chosen to be $1$. In particular, if $E$ is 
an aperiodic $\textrm{Bohr}_o$-set, then $Q(E) = Q(\bZ^{p+q})$.
\end{theorem}

Let us now explain how the application mentioned in the abstract of this paper follows from 
Theorem \ref{cons3}. Let $N = 3$ and consider the quadratic form (with $p = 2$, $q = 1$) 
given by
\[
Q(x,y,z) = x^2 + y^2 - z^2, \quad \textrm{for $(x,y,z) \in \bZ^3$}.
\]
Since every odd integer is a difference of two consequtive squares, this form satisfies $Q(\bZ^3) = \bZ$. In particular, for every 
finite subset $F \subset \bZ$, we can find a subset $F_o \subset \bZ^3$ of the same cardinality 
as $F$ such that $Q(F_o) = F$. We note that if $E_o \subset \bZ$ has positive upper Banach density, then 
$E := E_o \times E_o \times E_o \subset \bZ^3$ has positive upper Banach density as well. 
Hence, by Theorem \ref{cons3}, we conclude that for every positive integer $m$, there exists 
a positive integer $k$ such that for every finite set $F \subset \bZ$ with $|F| = m$,
we have
\[
k^2 F \subset Q((E_o - x) \times (E_o - y) \times (E_o - z)), \quad \textrm{for some $x,y,z \in E_o$},
\]
Finally, if $E_o \subset \bZ$ is a Bohr$_o$-set, then $E = E_o \times E_o \times E_o \subset \bZ^3$ is again 
a Bohr$_o$-set, so by the second assertion of Theorem \ref{cons3} we can in this case conclude that every
integer can be written on the form $u^2 + v^2 - w^2$ for some $u,v,w \in E_o$.

\subsection{Characteristic polynomials and their Galois groups}

Our second example concerns characteristic polynomials of integer square matrices with zero trace. Let $\textrm{Mat}_d(\bZ)$
denote the additive group of integer matrices, and define the subgroup $\Lambda_d < \textrm{Mat}_d(\bZ)$ by
\[
\Lambda_d = \{ a \in \text{Mat}_{d}(\bZ) \, : \, \trace(a)=0 \}.
\]
Given a matrix $a \in \Lambda_d$, we write $\cC(a)=\det(tI - a) \in \mathbb{Z}[t]$ to denote its characteristic 
polynomial. We note that the map $\cC : \Lambda_d \ra \bZ[t]$ satisfies $\cC(\gamma a \gamma^{-1}) = \cC(a)$ for all $a \in \Lambda_d$ and $\gamma \in \GL_d(\bZ)$. \\

The following theorem is an extension of a very recent result by the first author and A. Fish in the paper \cite{BjFish},
to which the current paper owes the initial ideas.  

\begin{theorem}
\label{cons1}
Let $d \geq 2$ and $E \subset \Lambda_d$ a set of positive upper Banach density. Let 
$m$ be a positive integer. Then there exists a positive integer $k$ with the property that
for every finite subset $F \subset \Lambda_d$ with $|F| = m$, we have
\[
\cC(kF) \subset \cC(E-b), \quad \textrm{for some $b \in E$}.
\]
If $E$ is an aperiodic Bohr set, then $k$ can be chosen to be $1$. In particular, if $E$ is 
an aperiodic $\textrm{Bohr}_o$-set, then $\cC(E) = \cC(\Lambda_d)$.
\end{theorem}

\begin{remark}
During the finalization of this paper, the authors were informed by A. Fish that he had independently proved the 
last assertion in Theorem \ref{cons1} (concerning aperiodic Bohr$_o$ sets); see \cite{Fish}.
\end{remark}

Given $a \in \Lambda_d$, we denote by $\bQ_a$ the field generated by the eigenvalues of $a$, or equivalently,
the splitting field of the polynomial $\cC(a)$. We note that 
\[
\bQ_{k a} = \bQ_a \qand \bQ_{\gamma a \gamma^{-1}} = \bQ_a, \quad \textrm{for all $k \in \bQ^*$ and $\gamma \in \GL_d(\bZ)$}.
\]
Given $P \in \bZ[t]$, we let $\cG(P)$ denote the Galois group (over $\bQ$) of the splitting field of $P$. Thus $\cG(\cC(a))$ is the Galois group of the field extension $\bQ_a/\bQ$. Since each $\cC(a)$ is a monic polynomial of degree $d$, we see that each $\cG(\cC(a))$ is a subgroup of the symmetric group $S_d$. Let $\cG_d$ denote the set of all possible subgroups $\cG(\cC(a)) < S_d$ as $a$ ranges over $\Lambda_d$. From the relations above, we see that
\[
\cG(\cC(ka)) = \cG(\cC(a)) \qand \cG(\cC(\gamma a \gamma^{-1})) = \cG(\cC(a)), \quad \textrm{for all $k \in \bN^*$ and $\gamma \in \GL_d(\bZ)$}.
\]

Let $F \subset \bZ^N$ be a finite set such that $\cG(F) = \cG_d$. Upon applying the map $\cG$ to the sets $\cC(kF)$ and $\cC(E-b)$ in Theorem \ref{cons1}, we have established the following corollary.
\begin{corollary}
Let $d \geq 2$ and suppose that $E \subset \Lambda_d$ is a set of positive upper Banach density. Then
there exists $b \in E$ such that $\cG_d \subset \cG(\cC(E-b))$, i.e. all possible Galois groups can be found in 
some translate of $E$.
\end{corollary}

This result should be compared with Gallagher's Theorem \cite{Gallagher} which asserts that "most" irreducible
monic polynomials with integer coefficients have Galois group $S_d$.

\subsection{Determinants of symmetric matrices}

Our final example involves determinants of symmetric integer matrices. We let $\Sym_d = \{ a \in \text{Mat}_{d}(\bZ)\text{ } |\text{ } a=a^t  \}$ denote the set of all symmetric $d \times d$ integer matrices.

\begin{theorem}
\label{cons2}
Let $d \geq 2$ and $E \subset \Sym_d$ a set of positive upper Banach density. Let 
$m$ be a positive integer. Then there exists a positive integer $k$ with the property that
for every finite subset $F \subset \Sym_d$ with $|F| = m$, we have
\[
k^d \, \det(F) \subset \det(E-b), \quad \textrm{for some $b \in E$}.
\]
If $E$ is an aperiodic Bohr set, then $k$ can be chosen to be $1$. In particular, if $E$ is 
an aperiodic $\textrm{Bohr}_o$-set, then $\det(E) = \bZ$.
\end{theorem}

In particular, let $E_o \subset \bZ$ be an aperiodic Bohr$_o$-set, and define 
\[
E = \Big\{
\left(
\begin{matrix}
x & z \\
z & y 
\end{matrix}
\right) 
\, : \, 
x,y,z \in E_o
\Big\} \subset \Sym_2.
\]
Then $E$ is a Bohr$_o$-set in $\Sym_2 \cong \bZ^3$ to which Theorem \ref{cons2} can applied to
yield the following corollary.

\begin{corollary}
Suppose that $E_o \subset \bZ$ is an aperiodic Bohr$_o$-set. Then,
\[
\{ xy - z^2 \, : \, x,y,z \in E_o \big\} = \bZ.
\]
\end{corollary}

\subsection{Invariant patterns in sets of positive upper Banach density}

We now turn to generalizing the three examples above. The main idea is that the functions presented in those examples (the quadratic forms, the characteristic polynomial map and the determinant map) are all invariant under certain linear actions. More specifically, the quadratic form $Q$ in Theorem~\ref{cons3} is preserved by $\SO(Q)(\bZ)$; the characteristic polynomial map $\cC$ and the Galois group map $\cG$ on $\Lambda_d$ are
both preserved by the conjugation action of $\SL_d(\bZ)$ on $\Lambda_d$; while the determinant map is preserved by the action of $\SL_d(\bZ)$ on $\Sym_d$ given by $\gamma \cdot a=\gamma a \gamma^t$. One of the main goals of this paper is to establish the following general result, to which the examples above apply (this will be verified in Section~\ref{appendix: BQ verification}). 

\begin{definition} A subgroup $\Gamma \leq \GL_N(\mathbb{R})$ is said to be \textit{strongly irreducible} if for every finite index subgroup $\Gamma' \leq \Gamma$, the standard representation of $\Gamma'$ on $\mathbb{R}^N$ is irreducible. We say that a Zariski connected real algebraic group $G$ has \textit{no compact factors} if every Zariski-continuous group homomorphism $\rho:G \to \GL_r(\bR)$ with bounded image is trivial (cf. Section 2 in \cite{BQ2}). To avoid confusion when necessary (in Section~\ref{appendix: BQ verification}), the usual Lie group theoretic compact factors will be referred to as the \textit{compact Lie group factors}.

\end{definition}

\begin{theorem}
\label{general}
Let $\Gamma < \GL_N(\bZ)$ be a non-trivial finitely generated strongly irreducible subgroup whose Zariski closure in $\GL_N(\bR)$ is a Zariski connected semisimple group with no compact factors. Let $Y$ be a set and suppose that $\Psi : \bZ^N \ra Y$ is a $\Gamma$-invariant function. For every $E \subset \bZ^N$ of positive upper Banach density and $m \geq 1$, there exists a positive integer 
$k$ with the property that whenever $F \subset \bZ^N$ is a finite set of cardinality $m$, then 
\[
\Psi(kF) \subset \Psi(E-b), \quad \textrm{for some $b \in E$}.
\]
Moreover, if $E \subset \bZ^N$ is an aperiodic $\textrm{Bohr}$-set, then $k$ can be chosen to be $1$. In particular, if $E$ is an aperiodic $\textrm{Bohr}_o$-set, then $\Psi(E) = \Psi(\bZ^N)$.
\end{theorem}

The following result is an immediate consequence of Theorem \ref{general}, and generalizes the main result in 
\cite{BjFish}.

\begin{corollary}
Let $\Gamma$ and $\Psi$ be as in Theorem \ref{general} and suppose that $E \subset \bZ^N$ has positive upper Banach density. Then there exists a positive integer $k$ such that
\[
\Psi(k \, \bZ^N) \subset \Psi(E-E).
\]
\end{corollary}

\subsection{Twisted multiple recurrence}
\label{twisted}
Theorem \ref{general} is derived from a "twisted" multiple recurrence result for ergodic $\bZ^N$-actions which 
we shall now state. Let $(X,\nu)$ be a Borel probability measure space, i.e. $X$ is a Borel subset of a compact and 
second countable space $\overline{X}$, and $\nu$ is a probability measure on the restriction of the Borel 
$\sigma$-algebra on $\overline{X}$ to $X$. Suppose that $\bZ^N$ acts on $X$ by Borel measurable bijections,
which preserve $\nu$. In this case we refer to $(X,\nu)$ as a \emph{$\bZ^N$-space}. We say that $(X,\nu)$
is \emph{ergodic} if whenever $B \subset X$ is a Borel set which is invariant under $\bZ^N$, then $B$ is either
a $\nu$-null set or a $\nu$-conull set.

We note that one can always associate to any $\bZ^N$-space a unitary representation $\pi_X$ of $\bZ^N$ on the 
Hilbert space $L^2(X,\nu)$ via 
\[
\big(\pi_X(a)f\big)(x) = f((-a) \cdot x), \quad \textrm{for $a \in \bZ^N$ and $f \in L^2(X,\nu)$}.
\]
Given a character $\chi$ on $\bZ^N$, we write 
\[
L^2(X,\nu)_\chi = \big\{ f \in L^2(X,\nu) \, : \,  \pi_X(a)f = \chi(a) f \big\} \subset L^2(X,\nu).
\]
We say that $\chi$ is a \emph{rational character} if there exists a positive integer $m$ such that $\chi(ma) = 1$ for all 
$a \in \bZ^N$. The set of all rational $\chi$ for which $L^2(X,\nu)_\chi$ is non-zero is called the 
\emph{rational spectrum} of the $\bZ^N$-space $(X,\nu)$. Since the constant function $1$ is fixed by 
$\pi_X$, we note that the rational spectrum always contains the trivial character $1$. If there are no other
elements in the rational spectrum, we say that the rational spectrum is \emph{trivial}.

\begin{theorem}
\label{main}
Let $(X,\nu)$ be an ergodic $\bZ^N$-space and suppose that $B$ is a Borel 
set in $X$. Let $\Gamma$ be as in Theorem \ref{general}. For every $\eps > 0$ and integer $m \geq 1$, there exists a positive integer $k$ with the property that whenever $a_1,\ldots,a_m$ are elements in $k \, \bZ^N$, then then there are $\gamma_1,\ldots, \gamma_m \in \Gamma$ such that 
\[
\nu\Big( \bigcap_{j=1}^m \big(\gamma_j a_j\big) \cdot B \Big) \geq \nu(B)^m - \eps.
\]
If the rational spectrum of the $\bZ^N$-space $(X,\nu)$ is trivial, then $k$ can be chosen to be $1$.
\end{theorem}

In Appendix I we outline how the following result can be deduced from Theorem \ref{main}. For 
the connection between trivial rational spectrum and aperiodic Bohr sets we refer the reader 
to Section \ref{Bohr}.

\begin{corollary}
\label{maincor}
Let $E \subset \bZ^N$ be a set of positive upper Banach density and $m \geq 1$. Let $\Gamma$ be as in 
Theorem \ref{general}. For every $\eps > 0$,
there exists a positive integer $k$ with the property that whenever $a_1,\ldots,a_m$ are elements in $k \, \bZ^N$,
then there are $\gamma_1,\ldots, \gamma_m \in \Gamma$ such that
\[
d^*\Big( \bigcap_{j=1}^m \big(E - \gamma_j a_j \big) \Big) \geq d^*(E)^m - \eps.
\]
If $E$ is an aperiodic Bohr set, then $k$ can be chosen to be $1$.
\end{corollary}

\subsection{Proof of Theorem \ref{general} using Corollary \ref{maincor}}
Let $Y$ be a set and $\Psi : \bZ^N \ra Y$ be a $\Gamma$-invariant function. Let $E \subset \bZ^N$ be a 
set of positive upper Banach density and $\eps > 0$ and let $m$ be a positive integer. By Corollary \ref{maincor} we can now find a positive integer $k$ with the property that for all $a_1,\ldots,a_m \in k \bZ^N$, there
are $\gamma_1,\ldots, \gamma_m \in \Gamma$ such that
\[
d^*\Big( E \cap \bigcap_{j=1}^m \big(E - \gamma_j a_j \big) \Big) \geq d^*(E)^{m+1} - \eps, 
\]
If $\eps < d^*(E)^{m+1}$, then the left hand side is positive, and we can find $b \in E$ such that
\[
b + \gamma_j a_j \in E, \quad \textrm{for every $j = 1,\ldots, m$}.
\]
In particular, $\Psi(a_j) = \Psi(\gamma_j a_j) \in \Psi(E-b)$ for each $j$. Since $a_1,\ldots, a_m \in k\bZ^N$ are arbitrary, this finishes the first part of the proof. Finally, by the second part of Corollary \ref{maincor}, if $E \subset \bZ^N$ is an aperiodic $\textrm{Bohr}_o$-set, then the integer $k$ above can be chosen to be $1$.

\subsection{A non-conventional mean ergodic theorem}

The proof of Theorem \ref{main} will use as a black box some recent deep results by Benoist and Quint 
from the papers \cite{BQ1} and \cite{BQ3}. The following definition will be useful.

\begin{definition}[BQ-pair]
Let $\Gamma < \GL_N(\bZ)$ be a non-trivial finitely generated 
irreducible subgroup and let $\mu$ be a finitely supported probability 
measure on $\Gamma$ whose support generates $\Gamma$ as a
semigroup. We say that $(\Gamma,\mu)$ is a \emph{BQ-pair} if the 
Zariski closure of $\Gamma$ is a Zariski-connected semisimple algebraic 
group with no compact factors. 
\end{definition}

Let $(\cH,\pi)$ be a unitary $\bZ^N$-representation on a separable Hilbert space $\cH$. Given a character
$\chi$ on $\bZ^N$, we define
\[
\cH_\chi = \big\{ v \in \cH \, : \, \pi(a)v = \chi(a)v, \enskip \textrm{for all $a \in \bZ^N$} \big\}.
\]
The \emph{rational spectrum} of $(\cH,\pi)$ is defined as the set of all rational characters on $\bZ^N$ for which
$\cH_\chi$ is non-zero. We say that the rational spectrum is \emph{trivial} if it is either empty or only consists of
the character $1$. Finally, we denote by $\cH_\textrm{rat}$ the linear span of $\cH_\chi$, as $\chi$ ranges over
the rational spectrum, and we write $\cH^G$ for the linear subspace of $\pi(G)$-invariant vectors in $\cH$.

Suppose that $\mu$ is a probability measure on $\Gamma$. We define
\[
\mu^{*j}(\gamma) = \sum \mu(\gamma_1) \cdots \mu(\gamma_j), \quad \textrm{for $j \geq 1$},
\]
where the sum is taken over all $j$-tuples $(\gamma_1,\ldots,\gamma_j)$ such that $\gamma_1 \cdots \gamma_j = 
\gamma$. \\

Our main technical result in this paper can now be stated as follows.

\begin{theorem}
\label{mainerg}
Let $(\Gamma,\mu)$ be a BQ-pair and let $(\cH,\pi)$ be a unitary $\bZ^N$-representation. For every $a \in \bZ^N$
and $v \in \cH$, the limit
\[
Q_a v := \lim_n \frac{1}{n} \sum_{j=1}^n \Big( \sum_{\gamma \in \Gamma} \mu^{*j}(\gamma) \pi(\gamma a) v \Big),
\]
exists in the norm topology on $\cH$. Furthermore, for every $\eps > 0$ and $v \in \cH$, there exists a positive integer $k$ with the property that whenever $a \in k \bZ^N$, then  
\[
\big\| Q_a v - P_{\textrm{rat}}v \big\| < \eps,
\]
where $P_{\textrm{rat}}$ denotes the orthogonal projection onto $\cH_{\textrm{rat}}$. If the rational spectrum of 
$(\cH,\pi)$ is trivial, then $Q_a$ coincides with the orthogonal projection onto the space of $\pi$-invariant vectors, for all $a \in \bZ^N \setminus \{0\}$.
\end{theorem}

\subsection{Proof of Theorem \ref{main} using Theorem \ref{mainerg}}
Let $(X,\nu)$ be an ergodic $\bZ^N$-space and let $(L^2(X,\nu,\pi_X)$ be the
associated $\bZ^N$-representation as in Subsection \ref{twisted}. Since $(X,\nu)$ is ergodic, we see that all $\pi_X$-invariant elements are ($\nu$-essentially) constant functions. Let $f \in L^2(X\nu)$ be a measurable function such that $0 \leq f \leq 1$, and define
\[
\big(Q^{(n)}_a f\big)(x) 
= 
\frac{1}{n} \sum_{j=1}^n 
\Big( 
\sum_{\gamma \in \Gamma} \mu^{*j}(\gamma) \big(\pi_X(\gamma a)f\big)(x)\Big), \quad \textrm{for $a \in \bZ^N$}.
\]
Let $f_{\textrm{rat}} = P_{\textrm{rat}}f$ and fix $m \in \bN$ and $\eps > 0$. By Theorem \ref{main} and the fact that $P_{\textrm{rat}}$ can be expressed as a conditional expectation (see $\S 7.4$ in \cite{EinsWard}), we know that:
\begin{itemize}
\item There exists a positive integer $k$ such that for all $a \in k\bZ^N \setminus \{0 \}$, and sufficiently large $n$, we have
\[
\|Q_{a}^{(n)}f\|_\infty \leq 1 \qand \|Q_{a}^{(n)} f - f_{\textrm{rat}}\| < \frac{\eps}{m}.
\]
\item We have
\[
0 \leq f_\textrm{rat} \leq 1 \qand \int_X f_\textrm{rat} \, d\nu = \int_X f \, d\nu.
\]
\item If the rational spectrum of $(X,\nu)$ is trivial, then $Q_a = P_\textrm{rat}$ and $Q_a f = \int_X f \, d\nu$ for
all non-zero $a \in \bZ^N$. In particular, the integer $k$ above can be chosen to be one. 
\end{itemize}
Now fix $a_1, \ldots, a_m \in k\bZ^n$. Hence, for some sufficiently large $n$, we have that
\[
\int_X Q^{(n)}_{a_1} f(x) \cdots Q^{(n)}_{a_m} f(x) \, d\nu(x) \geq \int_X f_{\textrm{rat}}^{m-s} f^s \, d\nu - \eps,
\]
where $s$ denotes the number of $a_i$'s equal to zero (note that $Q_0 f=f$). 
%Now we suppose that $f=\chi_B$ for some Borel set $B \subset X$.
 If $s > 0$ then, since $f^{m-s}_{\textrm{rat}} \in \cH_{\textrm{rat}}$, we have that $$\int_X f_{\textrm{rat}}^{m-s} f^s \, d\nu = \int_X f_{\textrm{rat}}^{m-s} f \, d\nu   = \int_X f_{\textrm{rat}}^{m-s+1} \, d\nu \geq \int_X f_{\textrm{rat}}^m \, d\nu. $$ Hence in either case we have that $$ \int_X Q^{(n)}_{a_1} f(x) \cdots Q^{(n)}_{a_m} f(x) \, d\nu(x) \geq \int_X f_{\textrm{rat}}^m \, d\nu - \eps \geq \left( \int_X f_{\textrm{rat}} \, d\nu \right)^m - \eps,$$ where in the last step we used H\"older's inequality. In particular, for all $a_1,\ldots,a_m \in k \bZ^N$, we can find $\gamma_1,\ldots,\gamma_m \in \Gamma$ such that
\[
\int_X f((-\gamma_1 a_1) \cdot x) \cdots f((-\gamma_m a_m) \cdot x) \, d\nu(x) \geq \Big( \int_X f \, d\nu \Big)^{m} - \eps,
\] which gives Theorem \ref{main}. 

\section{Proof of Theorem \ref{mainerg}}

Let $\bT^N$ denote the set of all homomorphisms from $\bZ^N$ into 
$S^1 = \big\{ z \in \bC^* \, : \, |z| = 1 \big\}$, and note that $\GL_N(\bZ)$
acts on $\bT^N$ by
\[
\big(\gamma^*\chi\big)(a) = \chi(\gamma^{-1} a), \quad \textrm{for $\chi \in \bT^N$ and $\gamma \in \GL_N(\bZ)$}.
\]
Given $\chi \in \bT^N$ and $\Gamma < \GL_N(\bZ)$, we define 
\[
\Gamma_{\chi} = \big\{ \gamma \in \Gamma \, : \, \gamma^* \chi = \chi \big\} < \Gamma.
\]
We recall that an element $\chi \in \bT^N$ is called \emph{rational} if there exists a positive integer $m$ such
that $\chi(ma) = 1$ for all $a \in \bZ^N$. 

\begin{lemma}
\label{ratfiniteindex}
Suppose that $\Gamma < \GL_N(\bZ)$ is infinite\footnote{This is satisfied for $\Gamma$ coming from a BQ pair: If $\Gamma$ is non-trivial and has Zariski connected Zariski closure, then it must be infinite.} and strongly irreducible and $\chi \in \bT^N$. Then the index 
$[\Gamma : \Gamma_{\chi}]$ is finite if and only if $\chi$ is rational.
\end{lemma}

\begin{proof}
Suppose that $[\Gamma : \Gamma_{\chi}]$ is finite, hence $\Gamma_{\chi}$ is non-trivial as $\Gamma$ is infinite. Then $\Lambda = \ker \chi < \bZ^N$ 
is a non-trivial $\Gamma_{\chi}$-invariant subgroup, and thus $V = \Lambda \otimes \bR < \bR^N$
is a non-trivial $\Gamma_\chi$-invariant linear subspace. By strong irreducibility of $\Gamma$, we have 
$V = \bR^N$, and thus $\Lambda$ must have finite index in $\bZ^N$. Let $m$ be the order
of $\bZ^N/\Lambda$. Then we have $\chi^m = 1$, and thus $\chi$ is rational. 

Suppose that $\chi$ is rational. Then $\Lambda = \ker \chi < \bZ^N$ has finite index, and 
$\Gamma$ acts on the finite set $\textrm{Im} \, \chi \cong \bZ^N/\Lambda$, which shows that 
$\Gamma_\chi = \Stab_\Gamma \Lambda$ has finite index in $\Gamma$.
\end{proof}

The main technical ingredient in the proof of Theorem \ref{mainerg} is the following deep result
by Benoist and Quint; see Th\'eorem\`e 1.3 in \cite{BQ1} and Corollary 1.10b) in \cite{BQ3}. If 
$\Gamma$ in addition contains an element with a dominant eigenvalue of multiplicity one, then
this result was established earlier by Bourgain, Furman, Lindenstrauss and Mozes; see
Theorem B in \cite{BFLM}.

\begin{theorem}
\label{BQ}
Let $(\Gamma,\mu)$ be a BQ-pair. For every $\chi \in \bT^N$ and $a \in \bZ^N \setminus \{0\}$, we have
\[
\lim_n \frac{1}{n} \sum_{j=1}^n \Big( \sum_{\gamma \in \Gamma} \chi(\gamma a) \mu^{*j}(\gamma) \Big)
=
0,
\]
if $[\Gamma : \Gamma_\chi] = \infty$, and 
\[
\lim_n \frac{1}{n} \sum_{j=1}^{n} \Big( \sum_{\gamma \in \Gamma} \chi(\gamma a) \mu^{*j}(\gamma) \Big)
=
\frac{1}{[\Gamma : \Gamma_\chi]} \sum_{\gamma \in \Gamma_\chi \backslash \Gamma} \chi(\gamma a),
\]
if $[\Gamma : \Gamma_\chi|] < \infty$.
\end{theorem}

\begin{remark}
We stress that Theorem \ref{BQ} is not explicated in neither of the papers \cite{BQ1} or \cite{BQ3}. Under the assumption that $(\Gamma,\mu)$ is a BQ-pair, Corollary 1.10b) in \cite{BQ3} asserts that for every $\chi \in \bT^N$, there exists a $\Gamma$-invariant Borel probability measure on $\nu_\chi$ on $\bT^N$, supported on the closure of the $\Gamma$-orbit of $\chi$, such that for every continuous function $f : \bT^N \ra \bC$, we have
\[
\frac{1}{n} \sum_{k=1}^n \Big( \sum_{\gamma \in \Gamma} f((\gamma^{*})^{-1} \chi) \mu^{*j}(\gamma) \Big) 
=
\int_{\bT^N} f \, d\nu_{\chi}. 
\] 
By Th\'eorem\`e 1.3 in \cite{BQ1}, $\nu_{\chi}$ is either the counting probability measure on a the (finite) $\Gamma$-orbit of $\chi$ in $\bT^N$ (in which case the index $[\Gamma : \Gamma_\chi]$ is finite), or it is equal to the Haar probability measure on $\bT^N$. We get Theorem \ref{BQ} by letting $f(\chi) = \chi(a)$ for $a \in \bZ^N$.
\end{remark}

Let $(\cH,\pi)$ be a unitary $\bZ^N$-representation on a separable Hilbert space $\cH$. 
Given $\chi \in \bT^N$, we define
\[
\cH_{\chi} = \big\{ v \in \cH \, : \, \pi(a) v  = \chi(a)v, \enskip \textrm{for all $a \in \bZ^N$}\big\}.
\]
One readily verifies that if $\chi_1$ and $\chi_2$ are distinct elements in $\bT^N$, then $\cH_{\chi_1}$
and $\cH_{\chi_2}$ are orthogonal subspaces in $\cH$. Since $\cH$ is separable, we conclude that
there is a possibly empty, but at most countable, set $\Omega \subset \bT^N$ such that $\cH_{\chi}$ is a
non-trivial subspace for $\chi \in \Omega$. The set of rational elements in $\Omega$ 
will be denoted by $\cR_N$, which we shall refer to as the \emph{rational spectrum} of $(\cH,\pi)$, and 
we write
\[
\cH_{\textrm{rat}} = \bigoplus_{\chi \in \cR_N} \cH_\chi \subset \cH,
\]
where the direct sum is taken in the Hilbert space sense. The following lemma is an immediate consequence
of the definitions above and the second assertion in Theorem \ref{BQ}, so we omit the proof.

\begin{lemma}
\label{rational}
For every $v \in \cH_{\textrm{rat}}$ and $a \in \bZ^N$, we have
\[
\lim_n \frac{1}{n} \sum_{j=1}^n \Big( \sum_{\gamma \in \Gamma} \mu^{*j}(\gamma) \pi(\gamma a)v \Big) = 
\sum_{\chi \in \cR_N}
\Big(
\frac{1}{[\Gamma : \Gamma_\chi]} \sum_{\gamma \in \Gamma_\chi \backslash \Gamma} \chi(\gamma a)
\Big)v_{\chi}
\]
where $v = \sum v_{\chi}$ and $v_\chi \in \cH_{\chi}$.
\end{lemma}

The full force of Theorem \ref{BQ} is released in the proof of the following lemma.

\begin{lemma}
\label{irrational}
For every $v \in \cH_{\textrm{rat}}^{\perp}$ and $a \in \bZ^N \setminus \{0\}$, we have
\[
\lim_n \frac{1}{n} \sum_{j=1}^n \Big( \sum_{\gamma \in \Gamma} \mu^{*j}(\gamma) \pi(\gamma a)v \Big) = 0.
\]
\end{lemma}

\begin{proof}
Let $v \in \cH_{\textrm{rat}}^{\perp}$ with $\|v\| = 1$. By Bochner's Theorem, there exists a Borel probability measure
$\eta$ on $\bT^N$ such that
\[
\langle \pi(a)v,v\rangle = \int_{\bT^N} \chi(a) \, d\eta(\chi), \quad \textrm{for all $a \in \bZ^N$}
.\]
We observe that, by an application of von-Neumann's mean ergodic theorem to the unitary $\bZ^N$-representation $(\cH, \pi_{\chi})$ given by $$\pi_{\chi}(a)v=\chi(a)^{-1} \pi(v) \quad  \textrm{ for $a \in \bZ^N$  and  $v \in \cH$},$$ we have that $\eta(\{\chi\}) = 0$ for every \emph{rational} $\chi \in \bT^N$. We note that
\[
\Big\|  
\frac{1}{n} \sum_{j=1}^n \Big( \sum_{\gamma \in \Gamma} \mu^{*j}(\gamma) \pi(\gamma a)v \Big) 
\Big\|^2
= 
\int_{\bT^N} 
\Big| 
\frac{1}{n} \sum_{j=1}^n \Big( \sum_{\gamma \in \Gamma} \mu^{*j}(\gamma) \chi(\gamma a) \Big)
\Big|^2 \, d\eta(\chi),
\]
for all $n$. By Lemma \ref{ratfiniteindex}, we have $[\Gamma : \Gamma_{\chi}] = \infty$ for every irrational 
$\chi$ and $[\Gamma : \Gamma_{\chi}] < \infty$ for every rational $\chi$. Hence, by Theorem \ref{BQ}, we 
conclude that the right-hand side above converges to
\[
\sum_{\chi \in \cR_N} 
\Big|
\frac{1}{[\Gamma : \Gamma_\chi]} \sum_{\gamma \in \Gamma_\chi \backslash \Gamma} \chi(\gamma a)
\Big|^2 \, \eta(\{\chi\}) = 0,
\]
since $\eta(\{\chi\}) = 0$ for all $\chi \in \cR_N$, which finishes the proof.
\end{proof}

Upon combining Lemma \ref{rational} and Lemma \ref{irrational}, we conclude that the limits
\[
Q_a v = \lim_n \frac{1}{n} \sum_{j=1}^n \Big( \sum_{\gamma} \mu^{*j}(\gamma) \, \pi(\gamma)v \Big)
\]
exist for every $v \in \cH$ and $a \in \bZ^N \setminus \{0\}$, and 
\[
Q_a v = \sum_{\chi \in \cR_N} 
\Big(
\frac{1}{[\Gamma : \Gamma_\chi]} \sum_{\gamma \in \Gamma_\chi \backslash \Gamma} \chi(\gamma a)
\Big)v_{\chi},
\]
where $P_{\textrm{rat}} v = \sum v_\chi$ and $v_\chi \in \cH_{\chi}$. In particular, if $\cR_N$ is trivial, i.e.
if $\cR_N$ is either empty or consists solely of the trivial character $1$, then $Q_a$ coincides with the 
orthogonal projection onto the closed subspace of $\pi(G)$-invariant elements in $\cH$. \\

The second assertion of Theorem \ref{mainerg} follows from the following lemma. Here, $P_\textrm{rat}$ denotes the orthogonal projection onto $\cH_{\textrm{rat}}$.

\begin{lemma}
\label{approx}
For every $\eps > 0$ and $v \in \cH$, there exists a positive integer $k$ such that 
\[
\big\| Q_a v - P_{\textrm{rat}}v \big\| < \eps, \quad \textrm{for all $a \in k \bZ^N \setminus \{0\}$}.
\]
\end{lemma}

\begin{proof}
Since $Q_a = 0$ on $\cH_{\textrm{rat}}^\perp$, it suffices to prove the lemma for $v \in \cH_{\textrm{rat}}$. 
Pick $\eps > 0$ and $v \in \cH_{\textrm{rat}}$ and choose a finite set $F \subset \bR_N$ such that
\[
\sum_{\chi \notin F} \|v_\chi\|^2 < \eps^2.
\]
Since $F$ is a finite set, we can find at least one positive integer $k$ such that $\chi(ka) = 1$ for all $\chi \in F$ and $a \in \bZ^N$. We note that this implies that $Q_{ka} v_\chi  = v_\chi$ for all $a \in \bZ^N$, and thus
\[
\big\|Q_a v - v \big\| = \big\| \sum_{\chi \notin F} Q_a v_\chi \big\| \leq  \big(\sum_{\chi \notin F} \big\| v_\chi\|^2\Big)^{\frac{1}{2}} < \eps,
\]
since $\|Q_a v\| \leq \|v\|$ for all $v \in \cH$ and $Q_a \cH_\chi \subset \cH_\chi$ for all $\chi \in \cR_N$.
\end{proof}

\section{Bohr sets and rational spectrum}
\label{Bohr}
We say that $E \subset \bZ^N$ is a \emph{Bohr set} if there exist a compact and second countable abelian group $K$ with Haar probability measure $m_K$, a homomorphism $\tau : \bZ^N \ra K$ with dense image, and a non-empty open set $U \subset K$ with $m_K(\overline{U}) = m_K(U)$ such that $E = \tau^{-1}(U)$. If $K$ is connected, we say that $E$ is \emph{aperiodic}, and if $U$ contains the identity element of $K$, we say that $B$ is an 
\emph{aperiodic Bohr$_o$} set. We note that if $B \subset \bZ^N$ is any aperiodic Bohr$_o$-set, then one can always
find another Bohr$_o$-set $C$ such that $C-C \subset B$.

\begin{example}
We give here an example of an aperiodic Bohr set in $\bZ$. Let $K = \bR/\bZ$ and suppose that $\alpha$ is an \emph{irrational} number. Then $\tau(a) = a \cdot \alpha \mod 1$ is a homomorphism from $\bZ$ into $K$ with dense image. Let $U \subset K$ be an open subset, e.g. an open interval. Then 
\[
B = \tau^{-1}(U) = \big\{ a \in \bZ \, : \, a \cdot \alpha \mod 1 \in U \big\} \subset \bZ
\]
is an aperiodic Bohr set in $\bZ$. More generally, for every integer $N$, we can form the homomorphism 
$\tau_N : \bZ^N \ra K^N$ defined by
\[
\tau(a_1,\ldots,a_N) = (\tau(a_1),\ldots,\tau(a_N)), \quad \textrm{for $(a_1,\ldots,a_N) \in \bZ^N$}.
\]
One can readily check that $\tau_N$ has dense image in $K^N$, and thus
\[
B \times \cdots \times B = \tau_N^{-1}(U \times \cdots \times U) \subset \bZ^N
\]
is an aperiodic Bohr$_o$-set in $\bZ^N$.
\end{example}

We can make $(K,m_K)$
into a $\bZ^N$-space with the $\bZ^N$-action defined by
\[
a \cdot x = x - \tau(a), \quad \textrm{for $x \in K$ and $a \in \bZ^N$}.
\]
We denote by $\pi_K$ the regular representation of $\bZ^N$ on $L^2(K,m_K)$ and given 
$\chi \in \bT^N$, we define 
\[
L^2(K,m_K)_\chi = \big\{ f \in L^2(K,m_K) \, : \, \pi_K(a)f = \chi(a) f \big\}.
\]
Let $\widehat{K}$ denote the dual of $K$. We can view $\eta \in \widehat{K}$ as an element
in $L^2(K,m_K)$ with the property that $\pi_K(a)\eta = \eta(\tau(a)) \eta$ for all $a \in \bZ^N$. 
Note that if $\eta_1, \eta_2 \in \widehat{K}$ satisfy $\eta_1 \circ \tau = \eta_2 \circ \tau$,
then $\eta_1 = \eta_2$ since the image of $\tau$ is dense. In particular, for every 
$\chi \in \bT^N$ of the form $\chi = \eta \circ \tau$, we have 
\[
L^2(K,m_K)_\chi = \bC \cdot \eta.
\]
Since all $\eta$ are orthogonal to each other in $L^2(K,m_K)$, and together span $L^2(K,m_K)$,
we conclude that
\[
L^2(K,m_K) = \bigoplus_{\eta \in \widehat{K}} L^2(K,m_K)_{\eta \circ \tau},
\]
where the direct sum is taken in the Hilbert space sense. Suppose that $\chi = \eta \circ \tau$ is 
rational, i.e. assume that there exists a positive integer $m$ such that $\chi^m = 1$. Then,
\[
\chi(a)^m = \eta(m\tau(a)) = \eta(\tau(ma)) = 1, \quad \textrm{for all $a \in \bZ^N$},
\]
and thus $\eta(k) = 1$ for all $k \in L$, where $L := \overline{\tau(m\bZ^N)} < K$, by continuity of $\eta$.
One readily shows that $L$ has finite index in $K$ and thus is an open subgroup of $K$. In particular,
if $K$ is \emph{connected}, then $L = K$, and $\eta = 1$, which establishes the following lemma.

\begin{lemma}
\label{bohrirrational}
Let $K$ be a compact and connected abelian group and suppose that $\tau : \bZ^N \ra K$ is a homomorphism
with dense image. Then the associated $\bZ^N$-space $(K,m_K)$ has trivial rational spectrum.
\end{lemma}

\section{Acknowledgements}

The authors have greatly benefited from discussions with Alexander Fish, Peter Hegarty and Jean-Fran\c{c}ois Quint. This work was carried out during the second author's visit to Chalmers University in September and October of 2015 funded by the University of Sydney's \textit{Philipp Hofflin International Research Travel Scholarship}, for which he is very grateful for. He would like to also thank the Department of Mathematical Sciences at Chalmers University for their hospitality. 
\section{Appendix I: Correspondence Principle}
\label{AppendixI}
We shall now explain how one can deduce Corollary \ref{maincor} from Theorem \ref{main}. The arguments 
in this section are nowadays rather standard, and can be traced back to the seminal paper \cite{Fur} by Furstenberg.

Suppose that $E \subset \bZ^N$. We may view $E$ as an element in the compact and second countable space $2^{\bZ^N}$ of all subsets of $\bZ^N$ equipped with the product topology, on which $\bZ^N$
acts by homeomorphisms via 
\[
a \cdot A = A - a, \quad \textrm{for $A \in 2^{\bZ^N}$ and $a \in \bZ^N$}.
\]
Let $X$ denote the closure of $\bZ^N \cdot E$ in $2^{\bZ^N}$. Then $X$ is again a compact and second countable
space, and
\begin{equation}
\label{openE}
V = \big\{ A \in X \, : \, 0 \in A \big\} \subset X,
\end{equation}
is a clopen (closed and open) subset of $X$. We note that $E = \big\{ a \in \bZ^N \, : \, a \cdot E \in V \big\}$. In 
other words, $E$ can be realized as the "hitting times" of the set $V$ of the $\bZ^N$-orbit of $E$ in $X$. \\

More generally, let $X$ be a compact and second space, equipped with an action of $\bZ^N$ by homeomorphisms.
Given a subset $U \subset X$ and $x \in X$, we define
\[
U_{x} = \big\{ a \in \bZ^N \, : \, a \cdot x \in U \big\} \subset \bZ^N.
\]
For instance, if $K$ is a compact and connected second countable group, $\tau : \bZ^N \ra K$ is a homomorphism
with dense image and $(K,m_K)$ denotes the associated $\bZ^N$-space defined in Section \ref{Bohr}, then for any non-empty open
subset $U \subset K$, we see that
\begin{equation}
\label{bohrset}
U_{0} = \big\{ a \in \bZ^N \, : \, \tau(a) \in -U \big\} = \tau^{-1}(-U) \subset \bZ^N
\end{equation}
is an aperiodic Bohr set. Since $K$ is connected, the $\bZ^N$-space $(K,m_K)$ has trivial rational spectrum by 
Lemma \ref{bohrirrational}.\\

Let $F_n = [-n,n]^N \subset \bZ^N$ and define the \emph{upper Banach density} of a subset $E \subset \bZ^N$ by
\[
d^*(E) = \sup
\Big\{ 
\limsup_n 
\frac{|E \cap (F_n + a_n)|}{|F_n|} 
\, : \, 
\textrm{$(a_n)$ is a sequence in $\bZ^N$} 
\Big\}.
\]
In particular, 
\[
d^*(E) \geq 
\limsup_n 
\frac{|E \cap F_n|}{|F_n|}, \quad \textrm{for all $E \subset \bZ^N$}.
\]
Let $\cP_{\bZ^N}(X)$ denote the (non-empty) convex set of $\bZ^N$-invariant Borel probability measures on the
compact and second countable space $X$. The following proposition can now be deduced from Theorem 1.1 in \cite{Fur}.

\begin{proposition}
\label{furst}
Suppose that $U \subset X$ is open and $x_o \in X$ has a dense $\bZ^N$-orbit in $X$. Then,
\[
d^*\Big( \bigcap_{a \in F} \big(U_{x_o} - a\big)\Big) \geq \nu\Big( \bigcap_{a \in F} a \cdot U\Big), 
\]
for every finite set $F \subset \bZ^N$ and $\nu \in \cP_{\bZ^N}(X)$. Furthermore, if $\nu(\overline{U}) = \nu(U)$ for all $\nu \in \cP_{\bZ^N}$, then 
\[
d^*(U_{x_o}) = \nu(U), \quad \textrm{for some \emph{ergodic} $\nu \in \cP_{\bZ^N}(X)$}.
\]
\end{proposition}

\subsection{Proof of Corollary \ref{maincor}} 
Suppose that $(X,\nu)$ is a compact and second countable $\bZ^N$-space, and let $U \subset X$ be a
non-empty open set such that $\nu(\overline{U}) = \nu(U)$ for all $\nu \in \cP_{\bZ^N}(X)$. 
For instance, we could choose:
\begin{itemize}
\item $X$ to be the orbit closure of the set $E \subset \bZ^N$ and $U = V$ as in \eqref{openE}. In this case, 
$U$ is a non-empty clopen set, and there exists an ergodic $\nu \in \cP_{\bZ^N}(X)$ such that 
\[
d^*(E) = d^*(U_{E}) = \nu(U).
\]
\item $K$ to be a compact, connected and second countable group, $\tau : \bZ^N \ra K$ a homomorphism with
dense image and $(X,\nu) = (K,m_K)$ the $\bZ^N$-space associated to $(K,\tau)$ as in Section \ref{Bohr}. In 
this case, $m_K$ is the unique $\bZ^N$-invariant Borel probability measure on $K$. In particular, for any open
subset such that $m_K(\overline{U}) = m_K(U)$, we have $d^*(\tau^{-1}(U)) = m_K(U)$, and $E = \tau^{-1}(U)$
is an aperiodic Bohr set.
\end{itemize}
Let $\Gamma < \GL_N(\bZ)$ and $a_1,\ldots,a_m \in \bZ^N$. In the first case above, Proposition \ref{furst} 
guarantees that
\[
d^*(E) = \nu(V) \qand d^*\Big( \bigcap_{j=1}^m \big(E-\gamma_j a_j)\Big) \geq \nu\Big(\bigcap_{j=1}^m (\gamma_j a_j) \cdot V \Big),
\]
for all $\gamma_1,\ldots,\gamma_m$, and in the second case above, Proposition \ref{furst} asserts that
\[
d^*(E) = m_K(U) \qand d^*\Big( \bigcap_{j=1}^m \big(\tau^{-1}(U)-\gamma_j a_j)\Big) \geq m_K \Big(\bigcap_{j=1}^m (\gamma_j a_j) \cdot U \Big),
\]
for all $\gamma_1,\ldots,\gamma_m$. \\

Let $\Gamma$ be as in Theorem \ref{general} and suppose that $(X,\nu)$ and $U$ are as in one of the two
examples above. Let $\eps > 0$ and let $m$ be a positive integer. By Theorem \ref{main}, there exist a positive 
integer $k$ with the property that whenever $a_1,\ldots, a_m \in k \bZ^N$, then 
\begin{equation}
\label{VE}
\nu\Big( \bigcap_{j=1}^m (\gamma_j a_j) \cdot U \Big) \geq \nu(U)^{m} - \eps, \quad \textrm{for some $\gamma_1,\ldots, \gamma_m \in \Gamma$}.
\end{equation}
Furthermore, if $(X,\nu)$ has trivial spectrum, as in the second example above (by Lemma \ref{bohrirrational}), then $k$ can be chosen to be $1$. \\

Upon combining the bounds above, we conclude that for all $a_1,\ldots, a_m \in k\bZ^N$, we have 
\[
d^*\Big( \bigcap_{j=1}^m \big(E - \gamma_j a_j \big) \Big) \geq d^*(E)^{m} - \eps,
\quad
\textrm{for some $\gamma_1,\ldots,\gamma_m \in \Gamma$}.
\]
In the case when $(X,\nu) = (K,m_K)$, the integer $k$ can be chosen to be $1$.

\section{Appendix II: Verifying the conditions for a BQ-pair}
\label{appendix: BQ verification}

We now verify that our examples satisfy the conditions of a BQ-pair. Note that in each of our examples we have a polynomial group homomorphism $\rho:G \to \SL_N(\mathbb{R})$ for some Zariski closed subgroup $G \leq \GL_d(\mathbb{R})$, which then defines an action of $G$ on $\bR^N$ given by $g \cdot v=\rho(g)v$. For example, in Theorem~\ref{cons1} we consider the adjoint representation $\Ad:\SL_d(\mathbb{R})) \to GL(\mathfrak{sl}_d(\bR))$, given by $$\Ad(g)v=gvg^{-1} \quad \text{ for } g\in \SL_d(\mathbb{R}) \text{ and } v \in \mathfrak{sl}_d(\bR),$$ where $\mathfrak{sl}_d(\bR)$ is the real vector space of real $d \times d$ traceless matrices. In other words, Theorem~\ref{cons1} is obtained from Theorem~\ref{general} by setting $\Gamma=\Ad(\SL_d(\bZ))$ (and identifying $\Lambda_d$ with $\bZ^{d^2-1}$).  The following Proposition ensures that such a representation $\rho$ preserves certain algebraic conditions in the definition of a BQ-pair.

\begin{proposition} Let $\rho:G \to \SL_N(\mathbb{R})$ be a polynomial homomorphism, where $G \subset \SL_d(\mathbb{R})$ is a Zariski connected semisimple Lie group with no compact algebraic factors. Then for $\Gamma \leq G$ Zariski dense, we have that the Zariski closure of $\rho(\Gamma)$ is a Zariski connected semisimple Lie group with no compact algebraic factors. 

\end{proposition}

\begin{proof} By Zariski-continuity, $\rho(G) \leq \overline{\rho(\Gamma)}^Z$ and in fact it is classical that $[\rho(G):\overline{\rho(\Gamma)}^Z]$ is finite (see for example Corollary 4.6.5 \cite{MorrisRatner}). Hence $\rho(G)$ being semisimple implies that  $\overline{\rho(\Gamma)}^Z$ also is. Again by Zariski continuity of $\rho$, we have that $\overline{\rho(\Gamma)}^Z$ is Zariski connected. Finally, suppose that $\kappa: \overline{\rho(\Gamma)}^Z \to GL_D(\mathbb{R})$ is a bounded algebraic group homomorphism (for some $D$), then so is $\kappa \circ \rho$ and so $\kappa(\rho(G))$ is the trivial subgroup. Thus $\rho(G) \leq \ker\kappa \leq \overline{\rho(\Gamma)}^Z$. But since $\ker\kappa$ is Zariski closed we have that it is equal to $\overline{\rho(\Gamma)}^Z$, so there are no compact factors. \end{proof}

\subsection{Algebro-geometric properties}

We now turn to determining the Zariski closures of $\SL_d(\bZ)$ and $\SO(Q)(\bZ)$ and verifying the required algebro-geometric properties (In this appendix, $Q$ will always denote a quadratic form as in Theorem~\ref{cons3}.). We first note the crucial fact that the groups $\SL_d(\bZ)$ and $\SO(Q)(\bZ)$ are, respectively, lattices in $\SL_d(\bR)$ and $\SO(Q)(\bR)$ (See Theorem 5.1.11 and Example 5.1.12 in \cite{MorrisArithmetic}). We also note that, as required by our main theorems, these lattices are finitely generated (See Theorem 4.7.10 in \cite{MorrisArithmetic} or Chapter IX in \cite{MargulisSS}). We will demonstrate below, via Borel's density theorem, that these lattices are Zariski dense. We remark the technicality that we use the following formulation of Borel's density theorem (not explicated in \cite{FurBorel} as it demands that $G$ is connected), which follows immediately from a combination of (4.5.1) in \cite{MorrisArithmetic} and (4.5.2) in \cite{MorrisRatner}. 

\begin{theorem}[Borel's density theorem] Let $G \leq \SL_{N}(\bR)$ be a Zariski connected semisimple Lie group (in particular, it has finitely many connected components) with no compact Lie group factors. Then any lattice in $G$ is Zariski dense. 

\end{theorem}

\begin{lemma} The group $\SL_d(\mathbb{R})$ is the Zariski closure of $\SL_d(\mathbb{Z})$ and is a Zariski-connected semisimple Lie group with no compact factors. 

\end{lemma}

\begin{proof} Zariski connectedness follows from the fact that $\SL_d(\mathbb{R})$ is connected in the Euclidean topology. The lack of compact factors follows from the much stronger classical fact that the only proper non-trivial normal (abstract) subgroup of $\SL_d(\mathbb{R})$ is its center (in particular, this also shows semisimplicity). Thus Borel's density theorem may be applied. \end{proof}

From now on, we identify $\SO(Q)(\bR)$ with $\SO(p,q)(\bR)$, as can be done via a linear change of coordinates.

\begin{lemma} For $p,q \geq 1$ with $p+q \geq 3$, the group $\SO(p,q)(\bR)$ is a Zariski-connected semisimple Lie group with no compact factors.  Moreover, the Zariski closure of $\SO(Q)(\bZ)$ is $\SO(Q)(\bR) \cong SO(p,q)(\bR)$. 
\end{lemma}

\begin{proof} Let $G=SO(p,q)(\bR)$ and let $G^o$ denote the connected (in the Euclidean topology) component of $SO(p,q)$. It follows from Problems 9 and 10 of Section 3 in Chapter 1 of \cite{OVbook} that $[G:G^o]=2$ and that $G^o$ is not Zariski closed. This implies that $G$ is the Zariski closure of $G^o$ and thus is Zariski connected. For $(p,q)\neq (2,2)$ it is well known (see for instance Appendix A in \cite{MorrisArithmetic}] that $G^o$ is simple as a Lie group and hence has no compact Lie group factors, while for $(p,q) \neq (2,2)$ we have that $G$ is a finite index quotient of $\SL_2(\bR) \times \SL_2(\bR)$ (see Appendix B in \cite{ZeeQuantum}) and thus is semisimple with no compact Lie group factors. In either case, we have that $G^o$ is contained in the kernel of all compact  (algebraic) factors of $G$. Hence, since $\SO^o(p,q)(\bR)$ is not Zariski closed, there are no non-trivial compact (algebraic) factors. Moreover, we may apply Borel's density theorem to obtain that all lattices (and hence $\SO(Q)(\bZ)$) are Zariski dense in $\SO(Q)(\bR) \cong G$. 
\end{proof}

\subsection{Irreducibility}

It now remains to check the strong irreducibility of the subgroups in our examples. Our first lemma shows that in fact it is enough to check the irreducibility of its Zariski closure.

\begin{lemma}[Irreducibility implies strong irreducibility] Suppose that $\Gamma \leq \GL_N(\bZ)$ is a subgroup such that its Zariski closure $G=\overline{\Gamma}^Z \leq GL_N(\bR)$ is Zariski connected and irreducible. Then $\Gamma$ is a strongly irreducible subgroup of $GL_N(\bR)$. 

\end{lemma}

\begin{proof} Let $V \leq \bR^N$ be a non-trivial subspace invariant under a finite index subgroup $\Gamma_0 \leq \Gamma$. Then $\overline{\Gamma_0}^Z$ also preserves $V$ and is a finite index Zariski closed subgroup of $G$, hence $G=\overline{\Gamma_0}^Z$ by Zariski connectedness of $G$. So $G$ preserves $V$ and so $V=\bR^N$, as required. \end{proof}

\begin{lemma} The adjoint action (i.e. action by conjugation) of $\SL_d(\bR)$ on $\mathfrak{sl}_d(\bR)$ is irreducible. 

\end{lemma}

\begin{proof} Let $W \leq \mathfrak{sl}_d(\bR)$ by a subspace that is invariant under the adjoint action. By differentiating, we see that $[\mathfrak{sl}_d(\bR), W]=W$, i.e. $W$ is an ideal. But it is well known that $\mathfrak{sl}_d(\bR)$ is simple.  \end{proof}

For a representation $G \curvearrowright V$ and $v\in V$, we let $\mathbb{R}[G]v$ denote the smallest $G$-invariant subspace containing $v$.

\begin{lemma} The action of $SL_d(\mathbb{R})$ on $\Sym_d$ given by $g.A=gAg^t$ is irreducible.

\end{lemma}

\begin{proof} Since each non-zero element of $\Sym_d$ is in the $G$-orbit of some diagonal matrix, it is enough to show that $\mathbb{R}[G]A=\Sym_d$ for each non-zero diagonal matrix $A$. All positive diagonal matrices (i.e. diagonal matrices with positive diagonal entires) are in the $G$-orbit of some positive constant multiple of the identity matrix, but the positive diagonal matrices span the space of all diagonal matrices. Thus it remains to show that if we fix a non-zero diagonal matrix $A=\operatorname{diag}(a_1,a_2, \ldots, a_d)$, then the space $\mathbb{R}[G]A$ contains a positive diagonal matrix. Note that the $G$-orbit of $A$ contains $$\operatorname{diag}(Ka_1,K^{-1/(d-1)}a_2, \ldots, K^{-1/(d-1)}a_d)\quad \text{ for all }K>0$$ and also $$\operatorname{diag}(a_{\sigma(1)}, \ldots, a_{\sigma(n)}) \quad \text{ for all } \sigma \in S_n,$$ which can be seen from the identity \begin{align*} \begin{pmatrix}
0 & 1 \\
-1 & 0
\end{pmatrix} \begin{pmatrix}
d_1 & 0 \\
0 & d_2
\end{pmatrix} \begin{pmatrix}
0 & -1 \\
1 & 0
\end{pmatrix} = \begin{pmatrix}
d_2 & 0 \\
0 & d_1
\end{pmatrix}. \end{align*}

Assuming (without loss of generality) that $a_1>0$, we see (by taking $K$ large enough) that the $G$-orbit of $A$ contains an element of the form $B_1=\operatorname{diag}(b_1,\ldots, b_n)$ where $b_1>d$ and $|b_k|<1$ for $k=2,\ldots, d$. The $G$-orbit of $A$ also contains $B_r=(b_{\sigma(1)}, \ldots, b_{\sigma(n)})$ where $\sigma$ is the transposition $(1r)$. Thus $B_1 + \cdots + B_d \in \mathbb{R}[G]A$ is a diagonal matrix with positive diagonal entries.

\end{proof}

\begin{lemma} For $p,q \geq 1$ with $p+q\geq 3$, the action of $SO(p,q)$ on $\mathbb{R}^{p+q}$ is irreducible.

\end{lemma} 
This will be deduced from the following general observation. 

\begin{lemma} Let $V$ and $W$ be vector spaces with $\operatorname{dim}W>1$ and let $H \leq GL(V), K \leq GL(W)$ be subgroups acting irreducibly on $V$ and $W$ respectively. Now suppose that $H \times K \leq G \leq GL(V \oplus W)$ is a subgroup such that $V \times \{0\}$ and $\{0\} \times W$ are not $G$-invariant. Then $G$ acts irreducibly on $V \oplus W$.

\end{lemma}

\begin{proof} Choose $x_0=(v_0,w_0) \in V \oplus W \setminus \{(0,0)\}$ and let $\mathbb{R}[G]x_0$ denote the smallest $G$-invariant subspace containing $x_0$. There exists $x_1=(v_1,w_1) \in \mathbb{R}[G]x_0$ such that $w_1 \neq 0$ (by non-invariance of $V \times \{0\}$). Now since $\operatorname{dim}W>1$ and $K$ acts irreducibly, there exists $k_1 \in K$ such that $k_1.w_1 \neq w_1$. Hence $$x_2:=x_1-(1,k_1).x_1=(0,w_2) \in \mathbb{R}[G]x_0$$ with $w_2=w_1-k_1 . w_1 \neq 0$. Since the action of $K$ is irreducible, we have $\{0\} \times W \leq \mathbb{R}[G]x_0$. But since $\{0 \} \times W$ is not $G$-invariant, there exists $(v_3,w_3) \in \mathbb{R}[G]x_0$ such that $v_3 \neq 0$. But $(v_3,0) = (v_3,w_3) - (0,w_3) \in \mathbb{R}[G]x_0$. So by irreducibility of $H$ we have that $V \times \{0\} \subset \mathbb{R}[G]x_0$. \end{proof}

The lemma applies (assuming $q \geq 2$) with $V=\mathbb{R}^p$, $W=\mathbb{R}^q$, $H=SO(p)$, $K=SO(q)$ and $G=SO(p,q)$. The non-invariance of $V$ and $W$ follow from considering a natural embedding $SO(1,1) \hookrightarrow SO(p,q)$ and using the fact that $SO(1,1)$ acts irreducibly on $\mathbb{R}^2$ (this can be seen by considering hyperbolic rotations).


\begin{thebibliography}{99}
\bibitem{BQ1}
Benoist, Y.; Quint, J-F. 
\emph{Mesures stationnaires et ferm\'es invariants des espaces homog\`enes}. (French) [Stationary measures and invariant subsets of homogeneous spaces] Ann. of Math. (2) 174 (2011), no. 2, 1111--1162. 

\bibitem{BQ2}Benoist, Y.; Quint, J-F. 
\emph{Stationary measures and invariant subsets of homogeneous spaces (II).} J. Amer. Math. Soc. 26 (2013), no. 3, 659--734. 

\bibitem{BQ3}
Benoist, Y.; Quint, J-F.
\emph{Stationary measures and invariant subsets of homogeneous spaces (III).} 
Ann. of Math. (2) 178 (2013), no. 3, 1017--1059.

\bibitem{BjFish}
Bj\"orklund, M; Fish, A.
\emph{Characteristic polynomial patterns in difference sets of matrices.} Preprint.

\bibitem{BFLM}
Bourgain, J.; Furman, A.; Lindenstrauss, E.; Mozes, S. 
\emph{Stationary measures and equidistribution for orbits of nonabelian semigroups on the torus.} 
J. Amer. Math. Soc. 24 (2011), no. 1, 231--280.

\bibitem{EinsWard}Einsiedler, Manfred; Ward, Thomas. 
\emph{Ergodic theory with a view towards number theory.} Graduate Texts in Mathematics, 259. Springer-Verlag London, Ltd., London, 2011. xviii+481 pp. ISBN: 978-0-85729-020-5

\bibitem{Fish}
Fish, A.
\emph{On Bohr sets of integer valued traceless matrices}. Preprint. 


\bibitem{FurBorel} Furstenberg, H. 
\emph{A note on Borel's density theorem.} Proc. Amer. Math. Soc. 55 (1976), no. 1, 209--212.

\bibitem{Fur}
Furstenberg, H.
\emph{Ergodic behavior of diagonal measures and a theorem of Szemer\'edi on arithmetic progressions.} 
J. Analyse Math. 31 (1977), 204--256. 

\bibitem{FKmulti} Furstenberg, H.; Katznelson, Y. 
\emph{An ergodic Szemer\'edi theorem for commuting transformations.} J. Analyse Math. 34 (1978), 275--291 (1979).

\bibitem{FKW}
Furstenberg, H.; Katznelson, Y.; Weiss, B. 
\emph{Ergodic theory and configurations in sets of positive density.} 
Mathematics of Ramsey theory, 184--198, Algorithms Combin., 5, Springer, Berlin, 1990.

\bibitem{Gallagher}
Gallagher, P. 
\emph{The large sieve and probabilistic Galois theory.} 
Analytic number theory (Proc. Sympos. Pure Math., Vol. XXIV, St. Louis Univ., St. Louis, Mo., 1972), 
pp. 91--101. Amer. Math. Soc., Providence, R.I., 1973.

\bibitem{Magyar} Magyar, \'A. 
\emph{On distance sets of large sets of integer points.} Israel J. Math. 164 (2008), 251--263.

\bibitem{MargulisSS} Margulis, G. A.
\emph{Discrete subgroups of semisimple Lie groups.} Ergebnisse der Mathematik und ihrer Grenzgebiete (3), 17. \textit{Springer-Verlag, Berlin,} 1991. x+388 pp. ISBN: 3-540-12179-X

\bibitem{MorrisArithmetic} Morris, D. W. 
\emph{Introduction to arithmetic groups.} Deductive Press, Place of publication not identified, 2015. xii+475 pp. ISBN: 978-0-9865716-0-2; 978-0-9865716-1-9 

\bibitem{MorrisRatner} Morris, D. W. 
\emph{Ratner's theorems on unipotent flows.} Chicago Lectures in Mathematics. \textit{University of Chicago Press}, Chicago, IL, 2005. xii+203 pp. ISBN: 0-226-53983-0; 0-226-53984-9 

\bibitem{OVbook}
Onishchik, A. L.; Vinberg, E. B. \emph{Lie groups and algebraic groups.} Translated from the Russian and with a preface by D. A. Leites. Springer Series in Soviet Mathematics. \textit{Springer-Verlag, Berlin,} 1990. xx+328 pp. ISBN: 3-540-50614-4 

\bibitem{ZeeQuantum}Zee, A. 
\emph{Quantum field theory in a nutshell.} Second edition. Princeton University Press, Princeton, NJ, 2010. xxviii+576 pp. ISBN: 978-0-691-14034-6 


\end{thebibliography}
\end{document}